\documentclass[preprint,12pt]{elsarticle}
\usepackage[a4paper,margin=1in]{geometry}
\usepackage[T1]{fontenc}
\usepackage[utf8]{inputenc}
\usepackage{lmodern}
\usepackage{amsmath,amssymb,amsfonts,amsthm,mathtools}
\usepackage{enumitem}

\usepackage[colorlinks=true,linkcolor=blue,citecolor=blue,urlcolor=blue]{hyperref}

\numberwithin{equation}{section}
\newtheorem{definition}{Definition}[section]
\newtheorem{remark}{Remark}[section]
\newtheorem{proposition}{Proposition}[section]
\newtheorem{lemma}{Lemma}[section]
\newtheorem{theorem}{Theorem}[section]
\newtheorem{corollary}{Corollary}[section]
\newtheorem{assumption}{Assumption}

\newcommand{\R}{\mathbb{R}}
\newcommand{\Rp}{\mathbb{R}_{+}}
\newcommand{\cF}{\mathcal{F}}
\newcommand{\cL}{\mathcal{L}}
\newcommand{\cP}{\mathcal{P}}
\newcommand{\cU}{\mathcal{U}}
\newcommand{\cR}{\mathcal{R}}
\newcommand{\E}{\mathbb{E}}
\newcommand{\Prob}{\mathbb{P}}
\newcommand{\1}{\mathbf{1}}
\newcommand{\wt}{\widetilde}

\begin{document}
	\begin{frontmatter}
		\title{Optimal control problem for reflected McKean--Vlasov stochastic differential equations with Poisson jumps}
		\author
              {Wenrui Lu}
         \ead{lwrstchhh@163.com}   	
      \author
          {Kai Wang\corref{cor1}}
\ead{wangkai050318@163.com}
		\affiliation
           {
			organization={School of Mathematics and Statistics, Anhui University of Finance and Economics},
			city={Bengbu 233030},
			state={Anhui},
			country={China}
	     	}
\cortext[cor1]{Corresponding author.}

\begin{abstract}
In this paper, we consider the optimal relaxed control problem for a class of one-dimensional reflected McKean--Vlasov stochastic differential equations with Poisson jumps. Due to the presence of the jump term, the state process generally belongs to the Skorokhod space $D([0,T],\Rp)$, which makes the proof of tightness and the passage to the limit more complicated. Under Lipschitz conditions and suitable growth conditions, we establish uniform moment estimates for the state process and the reflecting process. Then, by using Aldous' tightness criterion, the continuity of the Skorokhod map, and the stability results for stochastic integrals, we prove the existence of an optimal relaxed control. Furthermore, under the Roxin convexity condition, we prove the existence of a strict optimal control. In the general case, we show that relaxed controls can be approximated by a sequence of strict controls.
\end{abstract}


\end{frontmatter}
\section{Introduction}
McKean--Vlasov stochastic differential equations are stochastic dynamical systems whose coefficients depend on the distribution of the state process. They are also called distribution dependent stochastic differential equations or mean-field type stochastic differential equations. The idea is related to Vlasov's work on the mean-field theory of particle systems \cite{Vlasov1938}, while their probabilistic formulation is usually traced back to McKean's pioneering study on the relationship between nonlinear parabolic equations and Markov processes \cite{McKean1966}. In such models, the evolution of a single particle depends not only on its own state, but also on the probability distribution of the whole system. From the viewpoint of particle systems, McKean--Vlasov equations can be regarded as limiting models of weakly interacting particle systems as the number of particles tends to infinity, and are closely related to the theory of propagation of chaos \cite{Sznitman1991}. Throughout the sequel, we use the term McKean--Vlasov stochastic differential equations.

In recent years, McKean--Vlasov stochastic differential equations have been extensively studied in connection with well-posedness of solutions, particle-system approximation, propagation of chaos, nonlinear Fokker--Planck equations, and related partial differential equations. Many representative results have been obtained on the limits of weakly interacting particle systems, propagation of chaos, and well-posedness of McKean--Vlasov equations \cite{Sznitman1991,Buckdahn2017,Mishura2020}.

It is worth noting that Feng-Yu Wang and his collaborators have carried out a systematic series of studies in this field \cite{Wang2018,HuangWang2019,RenWang2019,HuangRenWang2021,HuangWang2022,HuangSongWang2022,Wang2023}, where the relevant equations are often formulated as distribution dependent stochastic differential equations. For example, Wang \cite{Wang2018} studied distribution dependent stochastic differential equations associated with Landau-type equations and established results on well-posedness, exponential contraction, and Harnack inequalities. Huang, Ren and Wang \cite{HuangRenWang2021} gave a systematic summary of several important developments on distribution dependent stochastic differential equations. Wang \cite{Wang2023} further investigated well-posedness and functional inequalities for distribution dependent reflecting stochastic differential equations. These works show that McKean--Vlasov stochastic differential equations have formed a rich theoretical framework and provide a theoretical basis for further study of reflected state-constrained problems and jump-type settings.

In the direction of optimal control, McKean--Vlasov systems have also attracted considerable attention. Since the coefficients depend on the state distribution, the control affects not only the evolution of the individual state, but also the overall dynamics of the system through the distributional term; hence the corresponding problem differs from classical stochastic control problems. Agram and {\O}ksendal \cite{AgramOksendal2023} studied optimal control problems for conditional McKean--Vlasov jump diffusions, and constructed a Markovian system that can be used for dynamic programming and Hamilton--Jacobi--Bellman (HJB) analysis through the stochastic Fokker--Planck equation satisfied by the conditional distribution. Agram, Pucci and {\O}ksendal \cite{AgramPucciOksendal2024} further studied impulse control problems for conditional McKean--Vlasov jump diffusions, and Agram and {\O}ksendal \cite{AgramOksendal2024} discussed the corresponding optimal stopping problem. In addition, the control theory of McKean--Vlasov systems includes regime-switching control \cite{ZhangLi2025}, convergence-rate analysis from finite-particle systems to McKean--Vlasov control problems \cite{Cardaliaguet2023}, and the relationship between mean-field control and mean-field games \cite{Carmona2016,Lacker2017}. These studies indicate that McKean--Vlasov control problems have become an important direction in stochastic control and mean-field theory.

On the other hand, reflected stochastic differential equations provide a natural tool for describing state-constrained systems. In many practical problems, the state process is subject to nonnegativity or domain constraints. For instance, risk reserves, inventory levels, population sizes, and workloads in queueing systems cannot cross a prescribed boundary arbitrarily. The role of the reflecting term is to exert a minimal regulation when the state process reaches the boundary, so that it remains in the admissible region. For reflected McKean--Vlasov control problems, Ma, Sun and Han \cite{MaSunHan2024} studied a stochastic control problem for one-dimensional reflected McKean--Vlasov SDEs and characterized the corresponding value function by a backward stochastic partial differential equation. Shao \cite{Shao2025} studied the optimal control problem for multidimensional reflected McKean--Vlasov SDEs and established the corresponding HJB equation on the Wasserstein space. For reflected McKean--Vlasov equations with jumps, Jarni, Missaoui and Ouknine \cite{Jarni2024} studied existence and uniqueness, stability, and propagation of chaos for the equation in time-dependent domains.

Existing studies have advanced reflected, jump, and McKean--Vlasov control problems from different perspectives. For example, Laayoun and Missaoui \cite{LaayounMissaoui2026} used the compactification method to prove the existence of an optimal relaxed control for one-dimensional reflected McKean--Vlasov SDEs. Under the Roxin convexity condition, they further proved that the optimal relaxed control can be attained by a strict control, and also showed that relaxed controls can be approximated by a sequence of strict controls. That work treats the continuous-path case driven by Brownian motion, where the proof of tightness and the passage to the limit are mainly carried out in the continuous path space. Jarni, Missaoui and Ouknine \cite{Jarni2024} studied well-posedness for reflected McKean--Vlasov equations with jumps, but did not involve optimal control. It follows that the existence of optimal controls for systems simultaneously involving reflection constraints, Poisson jump perturbations, and a relaxed control structure still requires further study.

Reflected McKean--Vlasov stochastic differential equations with jumps can simultaneously describe three mechanisms: the McKean--Vlasov structure captures mean-field interactions between individual states and the global distribution; the reflecting term captures nonnegativity or domain constraints; and the jump term models abrupt shocks or discontinuous perturbations. Compared with continuous diffusion models, models with jumps can more naturally describe abrupt information shocks in financial markets, large claims in insurance risk, batch arrivals in queueing networks, and external shocks in engineering systems. Because of the jump term, the state process generally no longer belongs to the space $C([0,T],\Rp)$ of continuous functions, but instead belongs to the Skorokhod space $D([0,T],\Rp)$. Therefore, the tightness arguments and passage-to-the-limit methods used in the continuous-path case cannot be applied verbatim; one has to handle tightness in the Skorokhod path space, the continuity of the reflecting map, and the stability of stochastic integrals.

Motivated by the above background, this paper studies the existence of optimal relaxed controls for a class of one-dimensional reflected McKean--Vlasov stochastic differential equations with Poisson jumps. We adopt a model setting in which the control variable acts on the drift term, and within this framework we analyze the influence of the jump term on the proof of existence for reflected McKean--Vlasov optimal controls. This setting preserves the basic averaging structure of the relaxed-control method and allows us to focus on the effect of Poisson jumps on tightness and the passage to the limit. The rationale for allowing the control to enter only the drift term will be discussed in the model formulation.

The main results of this paper are as follows. First, under suitable Lipschitz and growth conditions, we verify that the controlled reflected McKean--Vlasov state equation with jumps is well posed, by applying existing well-posedness results after averaging the drift with respect to the relaxed control, and we establish uniform moment estimates for the state process and the reflecting process. Second, by combining Aldous' tightness criterion with the continuity of the one-dimensional Skorokhod map, we prove the joint tightness of a minimizing sequence in the relaxed-control framework, and then obtain the existence of an optimal relaxed control. Third, by using estimates for compensated Poisson stochastic integrals and stability results for stochastic integrals with jumps, we complete the passage to the limit for the jump term, the distribution dependent term, and the relaxed control measure. Finally, under the Roxin convexity condition, we prove that the optimal relaxed control can be attained by a strict control, thereby obtaining the existence of a strict optimal control. In the general case, we show that relaxed controls can still be approximated by a sequence of strict controls, and consequently the strict control problem and the relaxed control problem have the same value. These results may be viewed as a jump extension of the existence theory for optimal controls of continuous-path reflected McKean--Vlasov systems.

The remainder of this paper is organized as follows. Section 2 introduces the preliminaries and the problem formulation, including the reflected problem, Poisson random measures, strict controls, relaxed controls, and the main assumptions. Section 3 presents the main results and their proofs, including well-posedness of the state equation, uniform moment estimates and tightness, the existence of an optimal relaxed control, and the relationship between strict controls and relaxed controls.

\section{Preliminaries and Problem Formulation}
Let $T>0$ be a fixed terminal time. Denote by
\(D([0,T],\Rp)\) the Skorokhod space consisting of all càdlàg functions from $[0,T]$ to $\Rp$, i.e., functions that are right-continuous and possess finite left limits at every point. This paper is concerned with reflected McKean--Vlasov stochastic differential equations with jumps. As their state processes are generally discontinuous, we conduct our analysis of the state process in the Skorokhod space $D([0,T],\Rp)$.

Let $A$ be a compact metric space, representing the set of control values, and $x\in\Rp$ be a deterministic initial state. All state processes in this paper start from the initial value $X_0=x$. 
Denote 
by $\cP_2(\Rp)$ the space of probability measures on $\Rp$ with finite second moment.

For $\mu,\eta\in \cP_2(\Rp)$, the second-order Wasserstein distance is defined by
\begin{equation}\label{eq:W2}
	W_2(\mu,\eta)
	=
	\left(
	\inf_{\pi\in\Pi(\mu,\eta)}
	\int_{\mathbb R_+\times\mathbb R_+}|x-y|^2\pi(dx,dy)
	\right)^{1/2}.
\end{equation}
where $\Pi(\mu,\eta)$ denotes the set of all probability measures with marginal distributions $\mu$ and $\eta$.

\subsection{The Skorokhod Reflection Problem}
\begin{definition}[One-dimensional Skorokhod reflection problem \cite{Skorokhod1961}]
Let $T>0$, and
\[
Y=(Y_t)_{t\in[0,T]}\in D([0,T],\R)
\]
be a given càdlàg function. If there exists a pair of processes
\[
(X,K)=(X_t,K_t)_{t\in[0,T]}\in D([0,T],\Rp)\times D([0,T],\Rp)
\]
satisfying the following conditions:

\begin{enumerate}[label=\arabic*.]
\item For every $t\in[0,T]$,
$
X_t=Y_t+K_t;
$
\item For every $t\in[0,T]$,
$
X_t\ge 0;
$
\item $K$ is a nondecreasing càdlàg process with
$
K_0=0;
$
and satisfies the minimal reflection condition
\[
\int_0^T \1_{\{X_s>0\}}\,dK_s=0,
\]
\end{enumerate}
then $(X,K)$ is called the Skorokhod reflection solution of $Y$ on the domain $[0,\infty)$.
\end{definition}

\begin{remark}
The process $K$ is called the reflecting process or the regulating process. The condition
\[
\int_0^T \1_{\{X_s>0\}}\,dK_s=0
\]
means that $K$ increases only when the state process $X$ is at the boundary $0$, thereby ensuring the minimality of the reflection.
\end{remark}

\subsection{Probability Space and Jump Noise}
Let $(\Omega,\mathcal F,\{\mathcal F_t\}_{t\in[0,T]},\mathbb P)$ be a filtered
probability space satisfying the usual conditions. Let
$B=(B_t)_{t\in[0,T]}$ be an $m$-dimensional standard Brownian motion with respect
to $\{\mathcal F_t\}_{t\in[0,T]}$. Let $N(dt,dz)$ be a Poisson random measure on
$[0,T]\times R_0$, independent of $B$, with compensator $\nu(dz)\,dt$ with respect
to $\{\mathcal F_t\}_{t\in[0,T]}$, where
$R_0:=\mathbb R\setminus\{0\}.$
Here $\nu(dz)$ is a L\'evy measure on $R_0$. The corresponding compensated
Poisson random measure is denoted by
\[
\widetilde N(dt,dz)=N(dt,dz)-\nu(dz)\,dt.
\]

In many stochastic control models, the drift term usually corresponds to an adjustable strategy, input or intervention intensity, whereas Brownian noise and Poisson jump noise more often reflect environmental uncertainty and external shocks, which the decision maker typically cannot directly intervene in.
Furthermore, from the viewpoint of relaxed control methods, allowing the control to enter only the drift term keeps the relaxed state equation in a well-defined mean structure. Specifically, once the strict control $u_t$ is replaced by a relaxed control measure $q_t(da)$, the drift term can be naturally written as
\[
\int_A b(t,x,\mu,a)q_t(da).
\]
If the control variable further enters the diffusion term or the jump term, the passage to the limit for the stochastic integral terms under relaxed controls is no longer merely an averaging over the control variable. Instead, one has to handle controlled quadratic variations, controlled jump measures, or corresponding stronger convexity conditions, which would significantly change the technical structure of the problem. Since this paper mainly focuses on the influence of the Poisson jump term on the proof of existence for reflected McKean--Vlasov optimal controls, especially tightness in the Skorokhod space, continuity of the reflecting map, and passage to the limit for compensated Poisson stochastic integrals, we adopt a model setting in which the control acts only on the drift term, in order not to obscure the core structure of the reflected mean-field jump system that this paper aims to reveal. That is, the control process affects the average evolutionary trend of the system state, while the diffusion term and the jump term describe external continuous random perturbations and sudden random shocks, respectively, whose intensities are not directly determined by the control variable.

Therefore, the drift coefficient, diffusion coefficient, and jump coefficient are written as
\[
b=b(t,x,\mu,a),\quad \sigma=\sigma(t,x,\mu),\quad \gamma=\gamma(t,x,\mu,z),
\]
where
\[
(t,x,\mu,a,z)\in [0,T]\times\Rp\times\cP_2(\Rp)\times A\times R_0.
\]

\subsection{Strict Controls}
\begin{definition}[Strict control]
A tuple $\alpha=(\Omega,\cF,\{\cF_t\}_{t\in[0,T]},\Prob,B,N,u,X,K)$ is called an admissible strict control if the following conditions hold:
\begin{enumerate}[label=\arabic*.]
\item$u=(u_t)_{t\in[0,T]}$ is an $A$-valued $\{\mathcal F_t\}_{t\in[0,T]}$-progressively measurable process;
\item $B=(B_t)_{t\in[0,T]}$ is an $m$-dimensional standard Brownian motion
with respect to $\{\mathcal F_t\}_{t\in[0,T]}$, and $N(dt,dz)$ is a Poisson
random measure on $[0,T]\times R_0$, independent of $B$, with
$\{\mathcal F_t\}_{t\in[0,T]}$-compensator $\nu(dz)\,dt$. The compensated
Poisson random measure is denoted by
$
\widetilde N(dt,dz)=N(dt,dz)-\nu(dz)\,dt;
$
\item $X$ and $K$ are $\{\cF_t\}$-adapted càdlàg processes, and $X,K\in D([0,T],\Rp)$;
\item Under the strict control $u=(u_t)_{t\in[0,T]}$, the reflected McKean--Vlasov stochastic differential equation with jumps considered in this paper can be formally written as
\[
\begin{cases}
 dX_t=b(t,X_t,\cL_{X_t},u_t)\,dt+\sigma(t,X_t,\cL_{X_t})\,dB_t \\
 \displaystyle\qquad\quad +\int_{R_0}\gamma(t,X_{t-},\cL_{X_{t-}},z)\,\widetilde{N}(dt,dz)+dK_t,\quad t\in[0,T],\\
 X_0=x,\quad X_t\ge 0,\quad t\in[0,T],\\
 \displaystyle K_0=0,\quad K\text{ is nondecreasing},\quad \int_0^T \1_{\{X_t>0\}}\,dK_t=0;
\end{cases}
\]
where $\cL_{X_t}$ denotes the distribution of the random variable $X_t$, $\widetilde{N}(dt,dz)$ is the compensated Poisson random measure, and $K$ is the reflecting process. 
\end{enumerate}
Equivalently, for every $t\in[0,T]$, the above equation can be written in integral form as
\begin{equation}\label{eq:strict-state}
	\begin{aligned}
		X_t
		=&\,x+\int_0^t b(s,X_s,\mathcal L_{X_s},u_s)\,ds
		+\int_0^t \sigma(s,X_s,\mathcal L_{X_s})\,dB_s  \\
		&+\int_0^t\int_{\mathbb R_0}
		\gamma(s,X_{s-},\mathcal L_{X_{s-}},z)\,\widetilde N(ds,dz)
		+K_t .
	\end{aligned}
\end{equation}
All admissible strict controls are denoted by $\cU$.
\end{definition}

For a strict control $\alpha\in\cU$, define the cost functional by
\begin{equation}\label{eq:strict-cost}
	J(\alpha)
	=
	\mathbb E\left[
	\int_0^T f(s,X_s,\mathcal L_{X_s},u_s)\,ds
	+\int_0^T c(s,X_s,\mathcal L_{X_s})\,dK_s
	+g(X_T,\mathcal L_{X_T})
	\right].
\end{equation}
The objective of the strict control problem is to find $\alpha^*\in\cU$ such that
\begin{equation}\label{eq:strict-value}
	J(\alpha^*)=\inf_{\alpha\in\mathcal U}J(\alpha).
\end{equation}

\subsection{Basic Assumptions}
To guarantee well-posedness, uniform moment estimates, and the subsequent proof of tightness for the reflected McKean--Vlasov stochastic differential equation with jumps, we impose the following assumptions.

\begin{assumption}[Control space and jump measure]
The control space $A$ is a compact metric space. Let $N(dt,dz)$ be a Poisson random measure defined on $[0,T]\times R_0$, and let its compensated Poisson random measure be $\widetilde{N}(dt,dz)=N(dt,dz)-\nu(dz)dt$, where $\nu(dz)$ is a Lévy measure on $R_0$ satisfying
\[
\int_{R_0}(1\wedge |z|^2)\nu(dz)<\infty.
\]
\end{assumption}

\begin{assumption}[Continuity and Lipschitz conditions on the coefficients]
The drift coefficient, diffusion coefficient, and jump coefficient are respectively
\[
b:[0,T]\times\Rp\times\cP_2(\Rp)\times A\to\R,
\]
\[
\sigma:[0,T]\times\Rp\times\cP_2(\Rp)\to\R^m,
\]
and
\[
\gamma:[0,T]\times\Rp\times\cP_2(\Rp)\times R_0\to\R.
\]
Here $b$ is continuous in $(t,x,\mu,a)$, $\sigma$ is continuous in $(t,x,\mu)$, and $\gamma$ is continuous in $(t,x,\mu)$ in the sense of $L^2(R_0,\nu)$.

Moreover, there exists a constant $C>0$ such that, for all $t\in[0,T]$, $a\in A$, $x,x'\in\Rp$, and $\mu,\mu'\in\cP_2(\Rp)$,
\[
\begin{aligned}
&|b(t,x,\mu,a)-b(t,x',\mu',a)|+|\sigma(t,x,\mu)-\sigma(t,x',\mu')| \\
&\quad +\left(\int_{R_0}|\gamma(t,x,\mu,z)-\gamma(t,x',\mu',z)|^2\nu(dz)\right)^{1/2}
\le C\bigl(|x-x'|+W_2(\mu,\mu')\bigr).
\end{aligned}
\]
\end{assumption}

\begin{assumption}[Growth conditions]
There exists a constant $C>0$ such that, for every $(t,x,\mu,a)\in [0,T]\times\Rp\times\cP_2(\Rp)\times A$,
\[
|b(t,x,\mu,a)|^2+|\sigma(t,x,\mu)|^2\le C\left[1+|x|^2+\int_{\Rp}|y|^2\mu(dy)\right].
\]
Moreover, the jump coefficient satisfies
\[
\left(\int_{R_0}|\gamma(t,x,\mu,z)|^2\nu(dz)\right)^2+
\int_{R_0}|\gamma(t,x,\mu,z)|^4\nu(dz)
\le C\left[1+|x|^4+\left(\int_{\Rp}|y|^2\mu(dy)\right)^2\right].
\]
\end{assumption}

\begin{assumption}[Conditions on the cost functions]
The cost functions
\[
f:[0,T]\times\Rp\times\cP_2(\Rp)\times A\to\R,
\]
\[
c:[0,T]\times\Rp\times\cP_2(\Rp)\to\R,
\]
and
\[
g:\Rp\times\cP_2(\Rp)\to\R
\]
are all continuous functions. Moreover, there exists a constant $C>0$ such that, for every $(t,x,\mu,a)\in[0,T]\times\Rp\times\cP_2(\Rp)\times A$,
\[
|f(t,x,\mu,a)|+|g(x,\mu)|\le C\left[1+|x|^2+\int_{\Rp}|y|^2\mu(dy)\right],
\]
and
\[
|c(t,x,\mu)|\le C\left[1+|x|+\left(\int_{\Rp}|y|^2\mu(dy)\right)^{1/2}\right].
\]
In particular, assume that the reflecting cost term is lower semicontinuous under the mode of convergence adopted in this paper. That is, if
\[
(X^n,K^n)\xrightarrow[n\to\infty]{D}(X,K)
\quad\text{in }D([0,T],\Rp)\times D([0,T],\Rp),
\]
and the corresponding moment estimates hold uniformly, then
\[
\E\left[\int_0^T c(s,X_s,\cL_{X_s})\,dK_s\right]
\le \liminf_{n\to\infty}\E\left[\int_0^T c(s,X_s^n,\cL_{X_s^n})\,dK_s^n\right].
\]
Furthermore, assume that the reflecting cost term is continuous along the reflecting process in the sense of uniform mean-square convergence. That is, if
\[
\E\left[\sup_{0\le t\le T}|X_t^n-X_t|^2+\sup_{0\le t\le T}|K_t^n-K_t|^2\right]\xrightarrow[]{n\to\infty}0,
\]
then
\[
\E\left[\int_0^T c(s,X_s^n,\cL_{X_s^n})\,dK_s^n\right]
\xrightarrow[]{n\to\infty}
\E\left[\int_0^T c(s,X_s,\cL_{X_s})\,dK_s\right].
\]
\end{assumption}

\begin{assumption}[Roxin convexity condition]
For every
\[
(t,x,\mu)\in[0,T]\times\Rp\times\cP_2(\Rp),
\]
the set
\[
S(t,x,\mu)=\{(b(t,x,\mu,a),f(t,x,\mu,a)):a\in A\}
\]
is a closed convex subset of $\R\times\R$.
\end{assumption}

\subsection{Relaxed Controls}
The basic idea of relaxed controls is to extend the point-valued control $u_t\in A$ to the probability measure $q_t(da)$ on $A$, that is, to allow the controller to choose a probability distribution at each time instead of a deterministic control value. For related ideas, see \cite{LaayounMissaoui2026,ElKarouiPengQuenez1997}.

Let $A$ be a compact metric space, and let $\cP(A)$ be the space of all probability measures on $A$, endowed with the topology of weak convergence. Let $\mathcal V\ $be the set of all nonnegative measures $q$ on $[0,T]\times A$ satisfying
\[
q([0,t]\times A)=t,\quad t\in[0,T].
\]
By the measure decomposition theorem, for every\(q\in\mathcal V\), there exists a family of $\cP(A)$-valued measurable maps $(q_t)_{t\in[0,T]}$ such that
\[
q(dt,da)=q_t(da)dt.
\]
Therefore, a relaxed control can be interpreted as a $\cP(A)$-valued process $(q_t)_{t\in[0,T]}$.
After equipping \(\mathcal V\) with the topology of stable convergence,\(\mathcal V\) is compact and metrizable \cite{JacodMemin1981}.
In particular, every strict control $u=(u_t)_{t\in[0,T]}$ can be naturally embedded into the relaxed control space.

\begin{definition}[Relaxed control]
A tuple $r=(\Omega,\cF,\{\cF_t\}_{t\in[0,T]},\Prob,B,N,q,X,K)$ is called an admissible relaxed control if it satisfies Conditions 2 and 3 in Definition 2.2 and the following conditions:
\begin{enumerate}[label=\arabic*.]
\item  $q=(q_t)_{t\in[0,T]}$ is a $\mathcal P(A)$-valued process
predictable with respect to $\{\mathcal F_t\}_{t\in[0,T]}$.
Equivalently, $q$ may be regarded as a predictable random measure on
$[0,T]\times A$ such that
\[
q(dt,da)=q_t(da)\,dt.
\]
In particular, for every Borel set $C\in\mathcal B(A)$, the process
\[
t\longmapsto q_t(C)
\]
is predictable with respect to $\{\mathcal F_t\}_{t\in[0,T]}$.
\item $K$ is a nondecreasing process with $K_0=0$. Under the relaxed control $q=(q_t)_{t\in[0,T]}$, $(X,K)$ satisfies the following reflected McKean--Vlasov stochastic differential equation with jumps:
\[
\begin{cases}
\displaystyle dX_t=\int_A b(t,X_t,\cL_{X_t},a)q_t(da)\,dt+
\sigma(t,X_t,\cL_{X_t})\,dB_t \\
\displaystyle\quad\qquad +\int_{R_0}\gamma(t,X_{t-},\cL_{X_{t-}},z)\,\widetilde{N}(dt,dz)+dK_t,\quad t\in[0,T],\\
X_0=x,\quad X_t\ge 0,\quad t\in[0,T],\\
\displaystyle K_0=0,\quad K\text{ is nondecreasing},\quad \int_0^T\1_{\{X_t>0\}}\,dK_t=0.
\end{cases}
\]
Equivalently, for every $t\in[0,T]$, the above equation can be written in integral form as
\begin{equation}\label{eq:relaxed-state}
	\begin{aligned}
		X_t
		=&\,x+\int_0^t\int_A b(s,X_s,\mathcal L_{X_s},a)q_s(da)\,ds
		+\int_0^t \sigma(s,X_s,\mathcal L_{X_s})\,dB_s\\
		&+\int_0^t\int_{\mathbb R_0}
		\gamma(s,X_{s-},\mathcal L_{X_{s-}},z)\,\widetilde N(ds,dz)
		+K_t .
	\end{aligned}
\end{equation}
\end{enumerate}
All admissible relaxed controls are denoted by $\cR$.
\end{definition}

For the relaxed control $r\in\cR$, define the cost functional by
\begin{equation}\label{eq:relaxed-cost}
	J(r)
	=
	\mathbb E\left[
	\int_0^T\int_A f(s,X_s,\mathcal L_{X_s},a)q_s(da)\,ds
	+\int_0^T c(s,X_s,\mathcal L_{X_s})\,dK_s
	+g(X_T,\mathcal L_{X_T})
	\right].
\end{equation}
The objective of the relaxed control problem is to find $r^*\in\cR$ such that
\begin{equation}\label{eq:relaxed-value}
	J(r^*)=\inf_{r\in\mathcal R}J(r).
\end{equation}
\subsection{Relationship Between Strict Controls and Relaxed Controls}
If there exists an $A$-valued adapted process $u=(u_t)_{t\in[0,T]}$ such that
\[
q_t(da)=\delta_{u_t}(da),
\]
where $\delta_{u_t}$ denotes the Dirac measure concentrated at $u_t$, then the relaxed control degenerates into a strict control. In this case,
\[
\int_A b(s,X_s,\cL_{X_s},a)q_s(da)=b(s,X_s,\cL_{X_s},u_s),
\]
and
\[
\int_A f(s,X_s,\cL_{X_s},a)q_s(da)=f(s,X_s,\cL_{X_s},u_s).
\]
Therefore, strict controls may be regarded as special cases of relaxed controls.

\subsection{Well-posedness of the State Equation and Well-definedness of the Control Problem}
\begin{proposition}[Well-posedness of the state equation]
Under Assumptions \textnormal{(A.1)}--\textnormal{(A.3)}, for any given relaxed control \(q\in\mathcal V\), the reflected McKean--Vlasov equation with jumps admits a unique adapted càdlàg solution $(X,K)$. In particular, when
\[
q_t(da)=\delta_{u_t}(da),
\]
the state equation under a strict control also admits a unique adapted càdlàg solution.
\end{proposition}

\begin{proof}
For a given relaxed control \(q\), define the averaged drift coefficient by 
\begin{equation}\label{eq:averaged-drift}
	\bar b_q(t,x,\mu)
	=
	\int_A b(t,x,\mu,a)q_t(da).
\end{equation}
 By the Lipschitz condition on \(b\), for all \(x,x'\in\mathbb R_+\) and \(\mu,\mu'\in\mathcal P_2(\mathbb R_+)\), \[ \begin{aligned} |\bar b_q(t,x,\mu)-\bar b_q(t,x',\mu')| &\leq \int_A |b(t,x,\mu,a)-b(t,x',\mu',a)|q_t(da) \\ &\leq C\bigl(|x-x'|+W_2(\mu,\mu')\bigr). \end{aligned} \] Moreover, by Jensen's inequality and the growth condition on \(b\), \[ |\bar b_q(t,x,\mu)|^2 \leq \int_A |b(t,x,\mu,a)|^2 q_t(da) \leq C\left(1+|x|^2+\int_{\mathbb R_+}|y|^2\mu(dy)\right). \] Hence the averaged drift coefficient \(\bar b_q\) satisfies the same Lipschitz and growth conditions as \(b\), with respect to \((x,\mu)\). Therefore, after the relaxed control \(q\) is fixed, the controlled state equation can be viewed as a reflected McKean--Vlasov SDE with jumps whose drift coefficient, diffusion coefficient and jump coefficient are  \(\bar b_q\), \(\sigma\) and  \(\gamma\), respectively. The equation thus falls within the framework of existing well-posedness results for reflected McKean--Vlasov SDEs with jumps\cite{Jarni2024}. Combining these results with the existence, uniqueness, and continuity of the one-dimensional Skorokhod map yields a unique adapted c\`adl\`ag solution \((X,K)\). When \(q_t(da)=\delta_{u_t}(da)\), the averaged drift reduces to \[ \bar b_q(t,x,\mu)=b(t,x,\mu,u_t), \] and hence the existence and uniqueness in the strict-control case follow as a special case.
\end{proof}

\begin{remark}
The above 
 result implies that, under Assumptions \textnormal{(A.1)}--\textnormal{(A.3)}, the admissible strict control set $\cU$ and the admissible relaxed control set $\cR$ are both nonempty. Indeed, taking any $a_0\in A$ and setting
\[
u_t\equiv a_0,
\quad q_t(da)=\delta_{a_0}(da),
\]
the proposition implies that the corresponding reflected McKean--Vlasov equation with jumps admits a unique adapted càdlàg solution $(X,K)$, thereby yielding an admissible strict control and an admissible relaxed control.

Moreover, under Assumptions \textnormal{(A.1)}--\textnormal{(A.4)}, the moment estimates for the state process and the reflecting process imply that the cost functional $J$ is well-defined. In general, the strict control set need not be compact, and when the corresponding convexity condition fails, a strict optimal control may not exist. Therefore, this paper introduces relaxed controls, extending the control variable from an $A$-valued process to a probability-measure-valued process on $A$, in order to obtain a suitable compactness structure.
\end{remark}

\section{Main Results and Proofs}
To prove the existence of an optimal relaxed control, we first need several auxiliary results concerning tightness of processes. To this end, take a minimizing sequence
\[
r^n=(\Omega^n,\cF^n,\{\cF_t^n\}_{t\in[0,T]},\Prob^n,B^n,N^n,q^n,X^n,K^n)\in\cR
\]
satisfying
\begin{equation}\label{eq:minimizing-sequence}
	\lim_{n\to\infty}J(r^n)=\inf_{r\in\mathcal R}J(r).
\end{equation}
For every $n\ge 1$, the process $(X^n,K^n)$ satisfies
\begin{equation}\label{eq:minimizing-state}
	\begin{aligned}
		X_t^n
		=&\,x+\int_0^t\int_A b(s,X_s^n,\mathcal L_{X_s^n},a)q_s^n(da)\,ds
		+\int_0^t\sigma(s,X_s^n,\mathcal L_{X_s^n})\,dB_s^n\\
		&+\int_0^t\int_{\mathbb R_0}
		\gamma(s,X_{s-}^n,\mathcal L_{X_{s-}^n},z)\,\widetilde N^n(ds,dz)
		+K_t^n ,\quad t\in[0,T],
	\end{aligned}
\end{equation}
and
\[
X_t^n\ge 0,\quad \int_0^T\1_{\{X_s^n>0\}}\,dK_s^n=0.
\]

We first prove uniform moment estimates and tightness for this minimizing sequence. Since this paper considers the case with jumps, the state process generally belongs to the Skorokhod space $D([0,T],\Rp)$, rather than to the space $C([0,T],\Rp)$ of continuous functions.
In what follows, $C$ denotes a generic positive constant which may change from line to line, but is independent of $n$. It may depend on the time horizon $T$, the initial value $x$, and the constants appearing in the assumptions.

\begin{lemma}[Uniform moment estimates and tightness]\label{lem:moment-tightness}
Under Assumptions \textnormal{(A.1)}--\textnormal{(A.3)}, let
\[
r^n=(\Omega^n,\cF^n,\{\cF_t^n\}_{t\in[0,T]},\Prob^n,B^n,N^n,q^n,X^n,K^n)
\]
be a sequence of admissible relaxed controls. Then there exists a constant $C>0$, independent of $n$, such that
\begin{equation}\label{eq:uniform-moment}
	\sup_{n\ge1}
	\mathbb E^{\mathbb P^n}\left[
	\sup_{0\le t\le T}|X_t^n|^4
	+
	\sup_{0\le t\le T}|K_t^n|^4
	\right]\le C.
\end{equation}
Moreover, $(X^n,K^n,B^n,N^n,q^n)$ is tight in the product space
\[
D([0,T],\Rp)\times D([0,T],\Rp)\times C([0,T],\R^m)\times M_p([0,T]\times R_0)\times V.
\]
\end{lemma}

\begin{proof}
Set
\[
\bar b_s^n=\int_A b(s,X_s^n,\cL_{X_s^n},a)q_s^n(da).
\]
Define the unreflected input process by
\begin{equation}\label{eq:Yn}
	\begin{aligned}
		Y_t^n
		=&\,x+\int_0^t \bar b_s^n\,ds
		+\int_0^t\sigma(s,X_s^n,\mathcal L_{X_s^n})\,dB_s^n\\
		&+\int_0^t\int_{\mathbb R_0}
		\gamma(s,X_{s-}^n,\mathcal L_{X_{s-}^n},z)\,\widetilde N^n(ds,dz).
	\end{aligned}
\end{equation}
Then
\[
X_t^n=Y_t^n+K_t^n.
\]
By the Lipschitz property of the one-dimensional Skorokhod map, there exists a constant $C>0$ such that
\begin{equation}\label{eq:skorokhod-estimate}
	\sup_{0\le t\le T}|X_t^n|^4
	+
	\sup_{0\le t\le T}|K_t^n|^4
	\le
	C\left(1+\sup_{0\le t\le T}|Y_t^n|^4\right).
\end{equation}
Hence it suffices to estimate $Y^n$.

By elementary inequalities,
\[
\begin{aligned}
\sup_{0\le t\le T}|Y_t^n|^4\le C\Bigg(&1+\sup_{0\le t\le T}\left|\int_0^t\bar b_s^n\,ds\right|^4
+\sup_{0\le t\le T}\left|\int_0^t\sigma(s,X_s^n,\cL_{X_s^n})\,dB_s^n\right|^4\\
&+\sup_{0\le t\le T}\left|\int_0^t\int_{R_0}\gamma(s,X_{s-}^n,\cL_{X_{s-}^n},z)\,\wt N^n(ds,dz)\right|^4\Bigg).
\end{aligned}
\]
First, by Hölder's inequality, Jensen's inequality, and the growth condition on $b$,
\[
\begin{aligned}
\E^{\Prob^n}\left[\sup_{0\le t\le T}\left|\int_0^t\bar b_s^n\,ds\right|^4\right]
&\le C\E^{\Prob^n}\left[\int_0^T|\bar b_s^n|^4\,ds\right]\\
&\le C\E^{\Prob^n}\left[\int_0^T\int_A |b(s,X_s^n,\cL_{X_s^n},a)|^4q_s^n(da)\,ds\right]\\
&\le C\left(1+\int_0^T\E^{\Prob^n}|X_s^n|^4\,ds\right)\\
&\le C\left(1+\int_0^T\E^{\Prob^n}\left[\sup_{0\le r\le s}|X_r^n|^4\right]ds\right).
\end{aligned}
\]
Second, by the Burkholder--Davis--Gundy inequality and the growth condition on $\sigma$,
\[
\begin{aligned}
\E^{\Prob^n}\left[\sup_{0\le t\le T}\left|\int_0^t\sigma(s,X_s^n,\cL_{X_s^n})\,dB_s^n\right|^4\right]
&\le C\E^{\Prob^n}\left[\left(\int_0^T|\sigma(s,X_s^n,\cL_{X_s^n})|^2\,ds\right)^2\right]\\
&\le C\left(1+\int_0^T\E^{\Prob^n}\left[\sup_{0\le r\le s}|X_r^n|^4\right]ds\right).
\end{aligned}
\]
Third, using the Burkholder--Davis--Gundy inequality for martingales with jumps, we have
\[
\begin{aligned}
&\E^{\Prob^n}\left[\sup_{0\le t\le T}\left|\int_0^t\int_{R_0}\gamma(s,X_{s-}^n,\cL_{X_{s-}^n},z)\,\wt N^n(ds,dz)\right|^4\right]\\
&\quad\le C\E^{\Prob^n}\left[\left(\int_0^T\int_{R_0}|\gamma(s,X_{s-}^n,\cL_{X_{s-}^n},z)|^2\nu(dz)ds\right)^2\right]\\
&\qquad +C\E^{\Prob^n}\left[\int_0^T\int_{R_0}|\gamma(s,X_{s-}^n,\cL_{X_{s-}^n},z)|^4\nu(dz)ds\right].
\end{aligned}
\]
By the growth condition on $\gamma$,
\[
\E^{\Prob^n}\left[\sup_{0\le t\le T}\left|\int_0^t\int_{R_0}\gamma(s,X_{s-}^n,\cL_{X_{s-}^n},z)\,\wt N^n(ds,dz)\right|^4\right]
\le C\left(1+\int_0^T\E^{\Prob^n}\left[\sup_{0\le r\le s}|X_r^n|^4\right]ds\right).
\]
Let
\[
\Phi_n(t)=\E^{\Prob^n}\left[\sup_{0\le r\le t}|X_r^n|^4\right].
\]
Combining the above estimates and using the estimate for the Skorokhod map yields
\begin{equation}\label{eq:Phi-Gronwall}
	\Phi_n(t)
	\le
	C\left(1+\int_0^t\Phi_n(s)\,ds\right).
\end{equation}
By Gronwall's lemma,
\[
\sup_{n\ge 1}\E^{\Prob^n}\left[\sup_{0\le t\le T}|X_t^n|^4\right]\le C.
\]
Furthermore, from the preceding estimate for $Y^n$, we have
\[
\E^{\Prob^n}\left[\sup_{0\le t\le T}|Y_t^n|^4\right]
\le C\left(1+\int_0^T\Phi_n(s)\,ds\right).
\]
Together with the uniform boundedness of $\Phi_n$, this gives
\[
\sup_{n\ge 1}\E^{\Prob^n}\left[\sup_{0\le t\le T}|Y_t^n|^4\right]\le C.
\]

On the other hand, by the explicit representation of the one-dimensional
Skorokhod reflection map, we have
\begin{equation}\label{eq:Kn-representation}
	K_t^n
	=
	\sup_{0\le s\le t}(-Y_s^n)^+.
\end{equation}
Since \(K^n\) is a nonnegative and nondecreasing process, it follows that
\[
\sup_{0\leq t\leq T}|K_t^n|
=K_T^n
=\sup_{0\leq t\leq T}(-Y_t^n)^+
\leq \sup_{0\leq t\leq T}|Y_t^n|.
\]
Therefore,
\begin{equation}\label{eq:Kn-estimate}
	\mathbb E^{\mathbb P^n}
	\left[
	\sup_{0\le t\le T}|K_t^n|^4
	\right]
	\le
	\mathbb E^{\mathbb P^n}
	\left[
	\sup_{0\le t\le T}|Y_t^n|^4
	\right]
	\le C.
\end{equation}
Combining this estimate with the preceding estimate for \(X^n\), we obtain
\[
\sup_n \mathbb{E}^{\mathbb{P}^n}\left[
\sup_{0\leq t\leq T}|X_t^n|^4
+
\sup_{0\leq t\leq T}|K_t^n|^4
\right]\leq C.
\]
Thus the uniform moment estimate is proved.

Next, we prove tightness. We first prove the tightness of $(Y^n)_{n\ge 1}$ in $D([0,T],\R)$ by using Aldous' tightness criterion. From the above fourth-moment estimate, there exists a constant $C>0$ such that
\[
\sup_{n\ge 1}\E^{\Prob^n}\left[\sup_{0\le t\le T}|Y_t^n|^4\right]\le C.
\]
Let $\varepsilon>0$. By Markov's inequality, for any $R>0$,
\[
\Prob^n\left(\sup_{0\le t\le T}|Y_t^n|>R\right)
\le \frac{1}{R^4}\E^{\Prob^n}\left[\sup_{0\le t\le T}|Y_t^n|^4\right]
\le \frac{C}{R^4}.
\]
Choosing $R>0$ sufficiently large such that $\frac{C}{R^4}<\varepsilon$, we obtain
\[
\Prob^n\left(Y_t^n\in[-R,R],\ \forall t\in[0,T]\right)
=\Prob^n\left(\sup_{0\le t\le T}|Y_t^n|\le R\right)\ge 1-\varepsilon.
\]
Since $[-R,R]$ is compact in $\R$, $(Y^n)_{n\ge 1}$ satisfies the compact-containment condition.

We now verify Aldous' condition. Let $\tau_n$ be an arbitrary $\{\cF_t^n\}$-stopping time taking values in $[0,T]$, and let $0\le \theta\le\delta$. Then
\[
Y_{\tau_n+\theta}^n-Y_{\tau_n}^n
=\int_{\tau_n}^{\tau_n+\theta}\bar b_s^n\,ds
+\int_{\tau_n}^{\tau_n+\theta}\sigma(s,X_s^n,\cL_{X_s^n})\,dB_s^n
+\int_{\tau_n}^{\tau_n+\theta}\int_{R_0}\gamma(s,X_{s-}^n,\cL_{X_{s-}^n},z)\,\wt N^n(ds,dz).
\]
Set
\[
\tau_n^\theta=(\tau_n+\theta)\wedge T,\quad 0\le \theta\le\delta.
\]
For the drift term, by the Cauchy--Schwarz and Jensen inequalities,
\[
\begin{aligned}
\E^{\Prob^n}\left[\left|\int_{\tau_n}^{\tau_n^\theta}\bar b_s^n\,ds\right|^2\right]
&\le \delta\E^{\Prob^n}\left[\int_{\tau_n}^{\tau_n^\theta}|\bar b_s^n|^2\,ds\right]\\
&\le \delta\E^{\Prob^n}\left[\int_{\tau_n}^{\tau_n^\theta}\int_A|b(s,X_s^n,\cL_{X_s^n},a)|^2q_s^n(da)\,ds\right].
\end{aligned}
\]
By the growth condition on $b$ and the preceding uniform moment estimate,
\[
\E^{\Prob^n}\left[\left|\int_{\tau_n}^{\tau_n^\theta}\bar b_s^n\,ds\right|^2\right]\le C\delta.
\]
For the Brownian term, by Itô's isometry,
\[
\E^{\Prob^n}\left[\left|\int_{\tau_n}^{\tau_n^\theta}\sigma(s,X_s^n,\cL_{X_s^n})\,dB_s^n\right|^2\right]
=\E^{\Prob^n}\left[\int_{\tau_n}^{\tau_n^\theta}|\sigma(s,X_s^n,\cL_{X_s^n})|^2\,ds\right].
\]
By the growth condition on $\sigma$ and the preceding uniform moment estimate,
\[
\E^{\Prob^n}\left[\left|\int_{\tau_n}^{\tau_n^\theta}\sigma(s,X_s^n,\cL_{X_s^n})\,dB_s^n\right|^2\right]\le C\delta.
\]
For the jump term, by the isometry formula for compensated Poisson integrals,
\[
\begin{aligned}
&\E^{\Prob^n}\left[\left|\int_{\tau_n}^{\tau_n^\theta}\int_{R_0}\gamma(s,X_{s-}^n,\cL_{X_{s-}^n},z)\,\wt N^n(ds,dz)\right|^2\right]\\
&\quad=\E^{\Prob^n}\left[\int_{\tau_n}^{\tau_n^\theta}\int_{R_0}|\gamma(s,X_{s-}^n,\cL_{X_{s-}^n},z)|^2\nu(dz)ds\right].
\end{aligned}
\]
By the growth condition on $\gamma$ and the preceding uniform moment estimate,
\[
\E^{\Prob^n}\left[\left|\int_{\tau_n}^{\tau_n^\theta}\int_{R_0}\gamma(s,X_{s-}^n,\cL_{X_{s-}^n},z)\,\wt N^n(ds,dz)\right|^2\right]\le C\delta.
\]
Therefore,
\begin{equation}\label{eq:Aldous-increment}
	\mathbb E^{\mathbb P^n}
	\left[
	|Y_{\tau_n+\theta}^n-Y_{\tau_n}^n|^2
	\right]
	\le C\delta.
\end{equation}
By Markov's inequality,
\[
\Prob^n\left(|Y_{\tau_n+\theta}^n-Y_{\tau_n}^n|>\varepsilon\right)\le \frac{C\delta}{\varepsilon^2}.
\]
Consequently, 
\begin{equation}\label{eq:Aldous-condition}
	\sup_n\sup_{\tau_n}
	\mathbb P^n
	\left(
	|Y_{\tau_n+\theta}^n-Y_{\tau_n}^n|>\varepsilon
	\right)
\xrightarrow[]{\delta\to0} 0.
\end{equation}
By Aldous' tightness criterion, $(Y^n)_{n\ge 1}$ is tight in $D([0,T],\R)$.

Finally, by the continuity of the one-dimensional Skorokhod map
\[
Y\mapsto (X,K)
\]
in the Skorokhod space, and since $X^n=\Gamma(Y^n)$ and $K^n=\Lambda(Y^n)$, where $\Gamma$ and $\Lambda$ denote the state component and reflection component of the one-dimensional Skorokhod map, respectively, namely
\[
\Lambda(Y)_t=\sup_{0\le s\le t}(-Y_s)^+,
\quad \Gamma(Y)_t=Y_t+\Lambda(Y)_t,
\]
it follows that $(X^n,K^n)_{n\ge1}$ is tight in $D([0,T],\Rp)\times D([0,T],\Rp)$. Moreover, since each $B^n$ is a standard Brownian motion, $(B^n)_{n\ge1}$ is tight in $C([0,T],\R^m)$. Since each $N^n$ is a Poisson random measure with intensity measure $ds\nu(dz)$, they have the same law as random variables with values in the space $M_p([0,T]\times R_0)$ of locally finite point measures. For the convergence theory of Poisson random measures and random measure spaces, see \cite[Chapters 3--4]{Kallenberg2017}. Therefore, $(N^n)_{n\ge1}$ is tight in $M_p([0,T]\times R_0)$. On the other hand, by the compactness of the relaxed control space \(\mathcal V\)  under the stable topology; see \cite[Lemma 3.5]{BahlaliGherbalMezerdi2011}; $(q^n)_{n\ge1}$ is tight in \(\mathcal V\) .

Combining these facts, the joint sequence $(X^n,K^n,B^n,N^n,q^n)$ is tight in the product space
\begin{equation}\label{eq:product-space}
	\Gamma
	:=
	D([0,T],\mathbb R_+)
	\times D([0,T],\mathbb R_+)
	\times C([0,T],\mathbb R^m)
	\times M_p([0,T]\times R_0)
	\times \mathcal V .
\end{equation}
The proof is complete.
\end{proof}

\subsection{Existence of an Optimal Relaxed Control}
\begin{theorem}[Existence of an optimal relaxed control]\label{thm:relaxed-existence}
Under Assumptions \textnormal{(A.1)}--\textnormal{(A.4)}, the relaxed control problem admits an optimal solution. That is, there exists $r^*\in\cR$ such that
\[
J(r^*)=\inf_{r\in\cR}J(r).
\]
\end{theorem}

\begin{proof}
Take a minimizing sequence
\[
r^n=(\Omega^n,\cF^n,\{\cF_t^n\}_{t\in[0,T]},\Prob^n,B^n,N^n,q^n,X^n,K^n)\in\cR
\]
satisfying
\[
\lim_{n\to\infty}J(r^n)=\inf_{r\in\cR}J(r).
\]
By the preceding uniform moment estimate and tightness result, the joint sequence
\[
\zeta^n=(X^n,K^n,B^n,N^n,q^n)
\]
is tight in the product space
\[
\Gamma:=D([0,T],\Rp)\times D([0,T],\Rp)\times C([0,T],\R^m)\times M_p([0,T]\times  R_0)\times \mathcal V\ .
\]
Therefore, by Prokhorov's theorem, there exists a subsequence, still denoted by $(\zeta^n)_{n\ge1}$, and a random variable
\[
\zeta=(X,K,B,N,q)
\]
on $\Gamma$ such that
\begin{equation}\label{eq:weak-convergence}
	\mathcal L(\zeta^n)\xrightarrow[n\to\infty]{w} \mathcal L(\zeta)
	\quad\text{on } \Gamma .
\end{equation}
By the Skorokhod representation theorem, there exists a probability space
$
(\hat\Omega,\hat\cF,\hat\Prob)
$
and random variables on this space,
\[
\hat\zeta^n=(\hat X^n,\hat K^n,\hat B^n,\hat N^n,\hat q^n),
\quad
\hat\zeta=(\hat X,\hat K,\hat B,\hat N,\hat q),
\]
such that, for every $n\ge1$,
\[
\cL(\hat\zeta^n)=\cL(\zeta^n),
\]
and
\begin{equation}\label{eq:skorokhod-convergence}
	\hat\zeta^n\to\hat\zeta,
	\quad \widehat{\mathbb P}\text{-a.s.}
\end{equation}
where the convergence is in the product space $\Gamma$.

Since $\hat\zeta^n$ and $\zeta^n$ have the same law, $(\hat X^n,\hat K^n)$ satisfies

\[
\begin{aligned}
\hat X_t^n={}&x+\int_0^t\int_A b(s,\hat X_s^n,\cL_{\hat X_s^n},a)\hat q_s^n(da)\,ds
+\int_0^t\sigma(s,\hat X_s^n,\cL_{\hat X_s^n})\,d\hat B_s^n\\
&\quad +\int_0^t\int_{R_0}\gamma(s,\hat X_{s-}^n,\cL_{\hat X_{s-}^n},z)\,\wt{\hat N}^n(ds,dz)+\hat K_t^n,
\quad t\in[0,T],
\end{aligned}
\]
where
\[
\wt{\hat N}^n(ds,dz)=\hat N^n(ds,dz)-\nu(dz)ds.
\]
Moreover,
\[
\hat X_t^n\ge0,
\quad
\int_0^T\1_{\{\hat X_s^n>0\}}\,d\hat K_s^n=0.
\]

We next show that the limiting process still satisfies the reflected McKean--Vlasov equation with jumps. 
We first handle the drift term. Let
\[
\hat I_b^n(t)=\int_0^t\int_A b(s,\hat X_s^n,\cL_{\hat X_s^n},a)\hat q_s^n(da)ds,
\]
and
\[
\hat I_b(t)=\int_0^t\int_A b(s,\hat X_s,\cL_{\hat X_s},a)\hat q_s(da)ds.
\]
Then
\[
\begin{aligned}
\hat I_b^n(t)-\hat I_b(t)
={}&\int_0^t\int_A\bigl[b(s,\hat X_s^n,\cL_{\hat X_s^n},a)-b(s,\hat X_s,\cL_{\hat X_s},a)\bigr]\hat q_s^n(da)ds\\
&+\int_0^t\int_A b(s,\hat X_s,\cL_{\hat X_s},a)(\hat q_s^n-\hat q_s)(da)ds\\
={}&:\hat I_{b,1}^n(t)+\hat I_{b,2}^n(t).
\end{aligned}
\]

For the first term, by the Lipschitz condition on $b$ with respect to $(x,\mu)$,
\[
|\hat I_{b,1}^n(t)|\le C\int_0^t\left(|\hat X_s^n-\hat X_s|+W_2(\cL_{\hat X_s^n},\cL_{\hat X_s})\right)ds.
\]
By the elementary property of the Wasserstein distance,
\[
W_2^2(\cL_{\hat X_s^n},\cL_{\hat X_s})\le \hat\E|\hat X_s^n-\hat X_s|^2.
\]
Combining the convergence $\hat X^n\xrightarrow[n\to\infty]{a.s.}\hat X$ in $D([0,T],\Rp)$, the preceding uniform moment estimates, and the dominated convergence theorem, we obtain
\[
\hat I_{b,1}^n(t)\xrightarrow[]{n\to\infty}0.
\]

We now handle the second term $\hat I_{b,2}^n(t)$. Since $b$ has linear growth, the function
\[
(s,a)\mapsto b(s,\hat X_s,\cL_{\hat X_s},a)
\]
is generally not bounded, and thus convergence cannot be obtained directly from $\hat q^n\xrightarrow[n\to\infty]{a.s.}\hat q$. We therefore use a truncation argument. Take $\chi_M\in C(\Rp;[0,1])$ such that
\[
\chi_M(x)=1,\quad 0\le x\le M,\quad \chi_M(x)=0,
\quad x\ge M+1.
\]
Define
\[
b_M(s,a)=\chi_M(\hat X_s)b(s,\hat X_s,\cL_{\hat X_s},a).
\]
For fixed $M>0$, $b_M$ is continuous and bounded in $a$. By the convergence of the relaxed control measures $\hat q^n\xrightarrow[n\to\infty]{a.s.}\hat q$ in the stable topology,
\[
\int_0^t\int_A b_M(s,a)\hat q_s^n(da)ds
\xrightarrow[]{n\to\infty}
\int_0^t\int_A b_M(s,a)\hat q_s(da)ds,
\qquad \hat\Prob\text{-a.s.}
\]
On the other hand, by the growth condition on $b$ and the uniform estimates,
\[
\int_0^t\int_A |b(s,\hat X_s,\cL_{\hat X_s},a)-b_M(s,a)|\hat q_s^n(da)ds
\le C\int_0^t\left(1+|\hat X_s|^2+\hat\E|\hat X_s|^2\right)\1_{\{|\hat X_s|>M\}}ds.
\]
The same estimate holds for $\hat q_s(da)$. Letting $M\to\infty$, the tail terms tend to zero by uniform integrability and the dominated convergence theorem. Hence
\[
\hat I_{b,2}^n(t)\xrightarrow[]{n\to\infty}0.
\]
Consequently,
\[
\hat I_b^n(t)\xrightarrow[]{n\to\infty} \hat I_b(t)=\int_0^t\int_A b(s,\hat X_s,\cL_{\hat X_s},a)\hat q_s(da)\,ds,
\quad \hat\Prob\text{-a.s.}
\]

Next, for the Brownian stochastic integral, the Lipschitz and growth conditions on $\sigma$ and the convergence $\hat X^n\xrightarrow[n\to\infty]{a.s.}\hat X$ yield
\[
\sigma(s,\hat X_s^n,\cL_{\hat X_s^n})\xrightarrow[]{n\to\infty} \sigma(s,\hat X_s,\cL_{\hat X_s})
\quad\text{in }L^2([0,T]\times\hat\Omega).
\]
Then, by the stochastic integral convergence theorem of Kurtz--Protter \cite[Theorem 2.2]{KurtzProtter1991},
\[
\int_0^t\sigma(s,\hat X_s^n,\cL_{\hat X_s^n})\,d\hat B_s^n
\xrightarrow[n\to\infty]{\mathbb P}
\int_0^t\sigma(s,\hat X_s,\cL_{\hat X_s})\,d\hat B_s,
\]
where the convergence holds in $D([0,T],\R)$.

We then handle the compensated Poisson integral. By the Lipschitz and growth conditions on $\gamma$,
\[
\hat\E\left[\int_0^T\int_{R_0}|\gamma(s,\hat X_{s-}^n,\cL_{\hat X_{s-}^n},z)-\gamma(s,\hat X_{s-},\cL_{\hat X_{s-}},z)|^2\nu(dz)ds\right]\xrightarrow[]{n\to\infty}0.
\]
Moreover, by the stability result of Kurtz--Protter for stochastic integrals with jumps \cite[Theorem 2.7]{KurtzProtter1991},
\[
\int_0^t\int_{R_0}\gamma(s,\hat X_{s-}^n,\cL_{\hat X_{s-}^n},z)\,\wt{\hat N}^n(ds,dz)
\xrightarrow[n\to\infty]{\mathbb P}
\int_0^t\int_{R_0}\gamma(s,\hat X_{s-},\cL_{\hat X_{s-}},z)\,\wt{\hat N}(ds,dz),
\]
where the convergence holds in $D([0,T],\R)$. Since
\[
\hat X^n\xrightarrow[]{n\to\infty}\hat X,
\quad
\hat K^n\xrightarrow[]{n\to\infty}\hat K,\quad \hat\Prob\text{-a.s.}
\]
together with the convergence of the three types of integral terms above, we let $n\to\infty$ and obtain
\begin{equation}\label{eq:limit-state}
	\begin{aligned}
		\hat X_t
		=&\,x+\int_0^t\int_A b(s,\hat X_s,\mathcal L_{\hat X_s},a)\hat q_s(da)\,ds
		+\int_0^t\sigma(s,\hat X_s,\mathcal L_{\hat X_s})\,d\hat B_s\\
		&+\int_0^t\int_{\mathbb R_0}
		\gamma(s,\hat X_{s-},\mathcal L_{\hat X_{s-}},z)\,\wt{\hat N}(ds,dz)
		+\hat K_t .
	\end{aligned}
\end{equation}
Moreover, by convergence in $D([0,T],\Rp)$ and the closedness of the one-dimensional Skorokhod map,
\[
\hat X_t\ge0,
\quad t\in[0,T],
\]
and
\[
\int_0^T\1_{\{\hat X_s>0\}}\,d\hat K_s=0.
\]
Therefore, the limiting tuple
\[
\hat r=(\hat\Omega,\hat\cF,\{\hat\cF_t\}_{t\in[0,T]},\hat\Prob,\hat B,\hat N,\hat q,\hat X,\hat K)
\]
is still an admissible relaxed control.

It remains to prove its optimality. Since $\hat\zeta^n$ and $\zeta^n$ have the same law,
\[
J(\hat r^n)=J(r^n).
\]
By the continuity and growth conditions of the running cost and terminal cost, together with the preceding uniform moment estimates, and by the lower semicontinuity of the reflection cost term under the convergence adopted in Assumption A.4, we obtain
\begin{equation}\label{eq:cost-lsc}
	J(\hat r)
	\le
	\liminf_{n\to\infty}J(\hat r^n).
\end{equation}
Therefore,
\[
J(\hat r)\le \liminf_{n\to\infty}J(\hat r^n)=\liminf_{n\to\infty}J(r^n)=\inf_{r\in\cR}J(r).
\]
On the other hand, by the definition of the infimum,
\[
J(\hat r)\ge \inf_{r\in\cR}J(r).
\]
Thus
\begin{equation}\label{eq:relaxed-optimality}
	J(\hat r)
	=
	\inf_{r\in\mathcal R}J(r).
\end{equation}
Hence $\hat r$ is an optimal relaxed control. The proof is complete.
\end{proof}

\subsection{Existence of a Strict Optimal Control}
\begin{corollary}[Existence of a strict optimal control]\label{cor:strict-existence}
Under Assumptions \textnormal{(A.1)}--\textnormal{(A.5)}, the strict control problem admits an optimal solution. That is, there exists an admissible strict control
\[
\alpha^*=(\Omega^*,\cF^*,\{\cF_t^*\}_{t\in[0,T]},\Prob^*,B^*,N^*,u^*,X^*,K^*)\in\cU
\]
such that
\[
J(\alpha^*)=\inf_{\alpha\in\cU}J(\alpha).
\]
Moreover,
\[
\inf_{\alpha\in\cU}J(\alpha)=\inf_{r\in\cR}J(r).
\]
\end{corollary}

\begin{proof}
By the existence theorem for optimal relaxed controls, there exists an optimal relaxed control
\[
\hat r=(\hat\Omega,\hat\cF,\{\hat\cF_t\}_{t\in[0,T]},\hat\Prob,\hat B,\hat N,\hat q,\hat X,\hat K)\in\cR
\]
satisfying
\[
J(\hat r)=\inf_{r\in\cR}J(r).
\]
We show that, under the Roxin convexity condition, one can construct a strict control with the same state equation and the same cost value as this optimal relaxed control.

For any $(t,\omega)\in[0,T]\times\hat\Omega$, define
\[
\hat c(t,\omega)=\left(\int_A b(t,\hat X_t(\omega),\cL_{\hat X_t},a)\hat q_t(\omega,da),
\int_A f(t,\hat X_t(\omega),\cL_{\hat X_t},a)\hat q_t(\omega,da)\right).
\]
By the Roxin condition, the set
\[
S(t,x,\mu)=\{(b(t,x,\mu,a),f(t,x,\mu,a)):a\in A\}
\]
is closed and convex in $\R\times\R$. Since $\hat q_t(\omega,\cdot)$ is a probability measure on $A$ and
\[
(b(t,\hat X_t(\omega),\cL_{\hat X_t},a),f(t,\hat X_t(\omega),\cL_{\hat X_t},a))
\in S(t,\hat X_t(\omega),\cL_{\hat X_t})
\]
for all $a\in A$, the stability of closed convex sets under integration with respect to probability measures implies that
\[
\hat c(t,\omega)\in S(t,\hat X_t(\omega),\cL_{\hat X_t}).
\]
Thus, for almost every $(t,\omega)$, there exists $a\in A$ such that
\[
\hat c(t,\omega)=\bigl(b(t,\hat X_t(\omega),\cL_{\hat X_t},a),f(t,\hat X_t(\omega),\cL_{\hat X_t},a)\bigr).
\]
By the measurable selection theorem, there exists an $A$-valued $\{\hat\cF_t\}$-adapted process
\[
\hat u=(\hat u_t)_{t\in[0,T]}
\]
such that, for almost every $(t,\omega)$,
\begin{equation}\label{eq:selection}
	\begin{aligned}
		\int_A b(t,\hat X_t,\mathcal L_{\hat X_t},a)\hat q_t(da)
		&=
		b(t,\hat X_t,\mathcal L_{\hat X_t},\hat u_t),\\
		\int_A f(t,\hat X_t,\mathcal L_{\hat X_t},a)\hat q_t(da)
		&=
		f(t,\hat X_t,\mathcal L_{\hat X_t},\hat u_t).
	\end{aligned}
\end{equation}
Since $\hat r$ is an admissible relaxed control, $(\hat X,\hat K)$ satisfies
\[
\begin{aligned}
\hat X_t={}&x+\int_0^t\int_A b(s,\hat X_s,\cL_{\hat X_s},a)\hat q_s(da)ds
+\int_0^t\sigma(s,\hat X_s,\cL_{\hat X_s})d\hat B_s\\
&\quad +\int_0^t\int_{R_0}\gamma(s,\hat X_{s-},\cL_{\hat X_{s-}},z)\,\wt{\hat N}(ds,dz)+\hat K_t.
\end{aligned}
\]
By the above choice of $\hat u$,
\[
\hat X_t=x+\int_0^t b(s,\hat X_s,\cL_{\hat X_s},\hat u_s)ds
+\int_0^t\sigma(s,\hat X_s,\cL_{\hat X_s})d\hat B_s
+\int_0^t\int_{R_0}\gamma(s,\hat X_{s-},\cL_{\hat X_{s-}},z)\,\wt{\hat N}(ds,dz)+\hat K_t.
\]
At the same time, $\hat X_t\ge0$, $t\in[0,T]$, and
\[
\int_0^T\1_{\{\hat X_s>0\}}\,d\hat K_s=0.
\]
Therefore
\[
\hat\alpha=(\hat\Omega,\hat\cF,\{\hat\cF_t\}_{t\in[0,T]},\hat\Prob,\hat B,\hat N,\hat u,\hat X,\hat K)
\]
is an admissible strict control, i.e., $\hat\alpha\in\cU$.

We now compare the cost functionals. By the construction of $\hat u$,
\[
\int_A f(s,\hat X_s,\cL_{\hat X_s},a)\hat q_s(da)=f(s,\hat X_s,\cL_{\hat X_s},\hat u_s).
\]
The reflecting cost term $\int_0^T c(s,\hat X_s,\cL_{\hat X_s})d\hat K_s$ and the terminal cost term $g(\hat X_T,\cL_{\hat X_T})$ do not depend on the control variable. Therefore,
\[
J(\hat\alpha)=J(\hat r).
\]
Since $\hat r$ is an optimal relaxed control,
\[
J(\hat\alpha)=J(\hat r)=\inf_{r\in\cR}J(r).
\]
On the other hand, every strict control can be regarded as a special relaxed control, namely
\[
q_t(da)=\delta_{u_t}(da).
\]
Hence
\[
\inf_{r\in\cR}J(r)\le \inf_{\alpha\in\cU}J(\alpha).
\]
Since $\hat\alpha\in\cU$, we also have
\[
\inf_{\alpha\in\cU}J(\alpha)\le J(\hat\alpha)=\inf_{r\in\cR}J(r).
\]
Combining the two inequalities,
\[
\inf_{\alpha\in\cU}J(\alpha)=\inf_{r\in\cR}J(r)=J(\hat\alpha).
\]
Thus $\hat\alpha$ is a strict optimal control. The proof is complete.
\end{proof}

\subsection{Approximation of Relaxed Controls by Strict Controls}
Corollary \ref{cor:strict-existence} shows that, when the Roxin convexity condition holds, the optimal relaxed control can be realized by some strict control. When the Roxin condition does not hold, however, a strict optimal control may fail to exist. The following theorem shows that, even in the general case, relaxed controls can still be approximated by a sequence of strict controls, so that the strict control problem and the relaxed control problem have the same value.

\begin{theorem}[Approximation of relaxed controls by strict controls]\label{thm:strict-approximation}
Under Assumptions \textnormal{(A.1)}--\textnormal{(A.4)}, for any admissible relaxed control
\[
r=(\Omega,\cF,\{\cF_t\}_{t\in[0,T]},\Prob,B,N,q,X,K)\in\cR,
\]
there exists a sequence of admissible strict controls
\[
\alpha^n=(\Omega,\cF,\{\cF_t\}_{t\in[0,T]},\Prob,B,N,u^n,X^n,K^n)\in\cU
\]
such that
\[
J(\alpha^n)\xrightarrow[]{n\to\infty} J(r).
\]
Consequently,
\begin{equation}\label{eq:value-equality}
	\inf_{\alpha\in\mathcal U}J(\alpha)
	=
	\inf_{r\in\mathcal R}J(r).
\end{equation}
\end{theorem}

\begin{proof}
Take any admissible relaxed control
\[
r=(\Omega,\cF,\{\cF_t\}_{t\in[0,T]},\Prob,B,N,q,X,K)\in\cR.
\]
By the chattering lemma, there exists a sequence of $A$-valued $\{\cF_t\}$-adapted processes $u^n=(u_t^n)_{t\in[0,T]}$ such that the corresponding Dirac-measure controls
\[
q_t^n(da):=\delta_{u_t^n}(da)
\]
satisfy
\begin{equation}\label{eq:chattering-convergence}
	q^n(dt,da)
	=
	\delta_{u_t^n}(da)dt
\xrightarrow[]{n\to\infty} 
	q(dt,da)
	=
	q_t(da)dt
\end{equation}
in the stable topology of $\mathcal V $.

For each $n\ge1$, let $(X^n,K^n)$ be the state process corresponding to the strict control $u^n$, namely
\[
\begin{aligned}
X_t^n={}&x+\int_0^t b(s,X_s^n,\cL_{X_s^n},u_s^n)ds
+\int_0^t\sigma(s,X_s^n,\cL_{X_s^n})dB_s\\
&\quad +\int_0^t\int_{R_0}\gamma(s,X_{s-}^n,\cL_{X_{s-}^n},z)\widetilde{N}(ds,dz)+K_t^n,
\quad t\in[0,T],
\end{aligned}
\]
and
\[
X_t^n\ge0,
\qquad
\int_0^T\1_{\{X_s^n>0\}}\,dK_s^n=0.
\]
Denote by $Y$ the unreflected input process corresponding to the relaxed control $q$,
\[
Y_t=x+\int_0^t\int_A b(s,X_s,\cL_{X_s},a)q_s(da)ds
+\int_0^t\sigma(s,X_s,\cL_{X_s})dB_s
+\int_0^t\int_{R_0}\gamma(s,X_{s-},\cL_{X_{s-}},z)\widetilde{N}(ds,dz),
\]
and by $Y^n$ the unreflected input process corresponding to the strict control $u^n$,
\[
Y_t^n=x+\int_0^t b(s,X_s^n,\cL_{X_s^n},u_s^n)ds
+\int_0^t\sigma(s,X_s^n,\cL_{X_s^n})dB_s
+\int_0^t\int_{R_0}\gamma(s,X_{s-}^n,\cL_{X_{s-}^n},z)\widetilde{N}(ds,dz).
\]
Then
\[
X^n=\Gamma(Y^n),\quad K^n=\Lambda(Y^n),
\]
and
\[
X=\Gamma(Y),\quad K=\Lambda(Y),
\]
where $\Gamma$ and $\Lambda$ denote the state component and the reflection component of the one-dimensional Skorokhod map, respectively.

By the Lipschitz property of the one-dimensional Skorokhod map,
\[
\E\left[\sup_{0\le t\le T}|X_t^n-X_t|^2+\sup_{0\le t\le T}|K_t^n-K_t|^2\right]
\le C\E\left[\sup_{0\le t\le T}|Y_t^n-Y_t|^2\right].
\]
We estimate the right-hand side. Let
\[
\Delta_n(t)=\E\left[\sup_{0\le s\le t}|X_s^n-X_s|^2\right].
\]
For the drift term, we decompose
\[
\begin{aligned}
&\int_0^t b(s,X_s^n,\cL_{X_s^n},u_s^n)ds
-\int_0^t\int_A b(s,X_s,\cL_{X_s},a)q_s(da)ds\\
&\quad=\int_0^t\bigl[b(s,X_s^n,\cL_{X_s^n},u_s^n)-b(s,X_s,\cL_{X_s},u_s^n)\bigr]ds\\
&\qquad+\int_0^t\left[b(s,X_s,\cL_{X_s},u_s^n)-\int_A b(s,X_s,\cL_{X_s},a)q_s(da)\right]ds.
\end{aligned}
\]
The first term is controlled by the Lipschitz condition on $b$ with respect to $(x,\mu)$:
\[
\begin{aligned}
&\E\left[\sup_{0\le r\le t}\left|\int_0^r\bigl[b(s,X_s^n,\cL_{X_s^n},u_s^n)-b(s,X_s,\cL_{X_s},u_s^n)\bigr]ds\right|^2\right]\\
&\quad\le C\int_0^t\E\left[|X_s^n-X_s|^2+W_2^2(\cL_{X_s^n},\cL_{X_s})\right]ds.
\end{aligned}
\]
By the property of the Wasserstein distance,
\[
W_2^2(\cL_{X_s^n},\cL_{X_s})\le \E|X_s^n-X_s|^2,
\]
and hence
\[
\begin{aligned}
&\E\left[\sup_{0\le r\le t}\left|\int_0^r\bigl[b(s,X_s^n,\cL_{X_s^n},u_s^n)-b(s,X_s,\cL_{X_s},u_s^n)\bigr]ds\right|^2\right]
\le C\int_0^t\Delta_n(s)ds.
\end{aligned}
\]
For the second term, the chattering lemma and the growth condition on $b$ give
\[
\E\left[\sup_{0\le r\le T}\left|\int_0^r\left(b(s,X_s,\cL_{X_s},u_s^n)-\int_A b(s,X_s,\cL_{X_s},a)q_s(da)\right)ds\right|^2\right]\xrightarrow[]{n\to\infty}0.
\]
Denote this term by $\varepsilon_n$, so that $\varepsilon_n\xrightarrow[]{n\to\infty}0$.

For the diffusion term, by the Burkholder--Davis--Gundy inequality, the Lipschitz condition on $\sigma$, and the property of the Wasserstein distance,
\[
\begin{aligned}
&\E\left[\sup_{0\le r\le t}\left|\int_0^r\bigl[\sigma(s,X_s^n,\cL_{X_s^n})-\sigma(s,X_s,\cL_{X_s})\bigr]dB_s\right|^2\right]\\
&\quad\le C\int_0^t\E\left[|X_s^n-X_s|^2+W_2^2(\cL_{X_s^n},\cL_{X_s})\right]ds\\
&\quad\le C\int_0^t\Delta_n(s)ds.
\end{aligned}
\]
For the jump term, by the BDG inequality for compensated Poisson integrals, the Lipschitz condition on $\gamma$, and the property of the Wasserstein distance,
\[
\begin{aligned}
&\E\left[\sup_{0\le r\le t}\left|\int_0^r\int_{R_0}\bigl[\gamma(s,X_{s-}^n,\cL_{X_{s-}^n},z)-\gamma(s,X_{s-},\cL_{X_{s-}},z)\bigr]\widetilde{N}(ds,dz)\right|^2\right]\\
&\quad\le C\E\left[\int_0^t\int_{R_0}|\gamma(s,X_{s-}^n,\cL_{X_{s-}^n},z)-\gamma(s,X_{s-},\cL_{X_{s-}},z)|^2\nu(dz)ds\right]\\
&\quad\le C\int_0^t\Delta_n(s)ds.
\end{aligned}
\]
Combining the above estimates and using the Lipschitz property of the Skorokhod map, we obtain
\[
\Delta_n(t)\le C\varepsilon_n+C\int_0^t\Delta_n(s)ds.
\]
By Gronwall's lemma, $\Delta_n(T)\xrightarrow[]{n\to\infty}0$. Therefore,
\[
\E\left[\sup_{0\le t\le T}|X_t^n-X_t|^2\right]\xrightarrow[]{n\to\infty}0,
\]
and the Skorokhod-map estimate also gives
\[
\E\left[\sup_{0\le t\le T}|K_t^n-K_t|^2\right]\xrightarrow[]{n\to\infty}0.
\]

We now prove convergence of the cost functionals. Decompose the running cost term as
\[
\begin{aligned}
&\int_0^T f(s,X_s^n,\cL_{X_s^n},u_s^n)ds
-\int_0^T\int_A f(s,X_s,\cL_{X_s},a)q_s(da)ds\\
&=\int_0^T\bigl[f(s,X_s^n,\cL_{X_s^n},u_s^n)-f(s,X_s,\cL_{X_s},u_s^n)\bigr]ds\\
&\quad+\int_0^T\left[f(s,X_s,\cL_{X_s},u_s^n)-\int_A f(s,X_s,\cL_{X_s},a)q_s(da)\right]ds.
\end{aligned}
\]
The first term converges to zero by the continuity and growth condition of $f$ with respect to $(x,\mu)$ and the $L^2$-convergence of $X^n$ to $X$. The second term converges to zero by the chattering lemma and the growth condition on $f$. Hence
\[
\E\left[\int_0^T f(s,X_s^n,\cL_{X_s^n},u_s^n)ds\right]
\xrightarrow[]{n\to\infty}
\E\left[\int_0^T\int_A f(s,X_s,\cL_{X_s},a)q_s(da)ds\right].
\]
For the terminal cost term, the continuity and growth condition on $g$ and the convergence $X_T^n\xrightarrow[n\to\infty]{L^2} X_T$ in $L^2$ imply
\[
\E[g(X_T^n,\cL_{X_T^n})]\to \E[g(X_T,\cL_{X_T})].
\]
If the cost functional contains the reflecting cost term, then by the continuity and growth condition on $c$ and the convergence $(X^n,K^n)\xrightarrow[]{n\to\infty}(X,K)$,
\[
\E\left[\int_0^T c(s,X_s^n,\cL_{X_s^n})\,dK_s^n\right]
\xrightarrow[]{n\to\infty}
\E\left[\int_0^T c(s,X_s,\cL_{X_s})\,dK_s\right].
\]
Therefore,
\begin{equation*}\label{eq:cost-convergence}
	J(\alpha^n)\xrightarrow[]{n\to\infty} J(r).
\end{equation*}

Finally, since every strict control can be regarded as a special relaxed control, i.e., $q_t(da)=\delta_{u_t}(da)$, we have
\[
\inf_{r\in\cR}J(r)\le \inf_{\alpha\in\cU}J(\alpha).
\]

On the other hand, by the approximation result above, for every $r\in\cR$ there exists a sequence of strict controls $(\alpha^n)_{n\ge1}\subset\cU$ such that $J(\alpha^n)\xrightarrow[]{n\to\infty} J(r)$. Hence
$
\inf_{\alpha\in\cU}J(\alpha)\le J(r).
$
Taking the infimum over $r\in\cR$ gives
$
\inf_{\alpha\in\cU}J(\alpha)\le \inf_{r\in\cR}J(r).
$
Combining these two inequalities, we obtain
$
\inf_{\alpha\in\cU}J(\alpha)=\inf_{r\in\cR}J(r).
$
The proof is complete.
\end{proof}

\appendix
\renewcommand{\thelemma}{\Alph{section}.\arabic{lemma}}
\renewcommand{\theremark}{\Alph{section}.\arabic{remark}}
\renewcommand{\theproposition}{\Alph{section}.\arabic{proposition}}
\renewcommand{\thedefinition}{\Alph{section}.\arabic{definition}}
\renewcommand{\theequation}{\Alph{section}.\arabic{equation}}

\section{Auxiliary Theorems}
\begin{lemma}[Aldous' tightness criterion \cite{Aldous1978}]
Let $(Y^n)_{n\ge1}$ be a sequence of stochastic processes taking values in $D([0,T],\R)$. Suppose that the following two conditions are satisfied:
\begin{enumerate}[label=\arabic*.]
\item For every $\varepsilon>0$, there exists a compact set $K_\varepsilon\subset\R$ such that
\[
\inf_{n\ge1}\Prob^n(Y_t^n\in K_\varepsilon,\ \forall t\in[0,T])\ge1-\varepsilon.
\]
\item For every $\varepsilon>0$,
\[
\lim_{\delta\downarrow0}\sup_{n\ge1}\sup_{\tau_n}\sup_{0\le\theta\le\delta}
\Prob^n(|Y_{\tau_n+\theta}^n-Y_{\tau_n}^n|>\varepsilon)=0,
\]
where the supremum is taken over all stopping times $\tau_n$ taking values in $[0,T]$, and $Y_{\tau_n+\theta}^n$ is understood as $Y_{(\tau_n+\theta)\wedge T}^n$.
\end{enumerate}
Then $(Y^n)_{n\ge1}$ is tight in the Skorokhod space $D([0,T],\R)$.
\end{lemma}

\begin{lemma}[One-dimensional Skorokhod map \cite{Skorokhod1961}]\label{lem:skorokhod-map}
Let $Y\in D([0,T],\R)$. Define
\[
\Lambda(Y)_t=\sup_{0\le s\le t}(-Y_s)^+,
\quad t\in[0,T],
\]
and set $\Gamma(Y)_t=Y_t+\Lambda(Y)_t$. Then $X=\Gamma(Y)$ and $K=\Lambda(Y)$ constitute the one-dimensional Skorokhod reflection solution of $Y$ on the domain $[0,\infty)$. That is,
\[
X_t=Y_t+K_t,
\quad X_t\ge0,
\quad K_0=0,
\]
$K$ is a nondecreasing càdlàg process and satisfies the minimal reflection condition
\[
\int_0^T\1_{\{X_s>0\}}\,dK_s=0.
\]
Moreover, the map $Y\mapsto(\Gamma(Y),\Lambda(Y))$ is continuous in the Skorokhod space. In particular, if
\[
Y^n\xrightarrow[]{n\to\infty} Y\quad\text{in }D([0,T],\R),
\]
then $\Gamma(Y^n)\xrightarrow[]{n\to\infty}\Gamma(Y)$ and $\Lambda(Y^n)\xrightarrow[]{n\to\infty}\Lambda(Y)$ in the corresponding Skorokhod spaces.

Furthermore, under the uniform norm there is a Lipschitz-type estimate: there exists a constant $C>0$ such that, for all $Y^1,Y^2\in D([0,T],\R)$,
\[
\sup_{0\le t\le T}|\Gamma(Y^1)_t-\Gamma(Y^2)_t|
+\sup_{0\le t\le T}|\Lambda(Y^1)_t-\Lambda(Y^2)_t|
\le C\sup_{0\le t\le T}|Y_t^1-Y_t^2|.
\]
\end{lemma}

\begin{lemma}[Chattering lemma \cite{ElKarouiNguyenJeanblanc1987}]
Let $A$ be a compact metric space, and let \(\mathcal V\) be the relaxed control space endowed with the stable topology. For every relaxed control
\[
q(dt,da)=q_t(da)dt\in \mathcal V,
\]
there exists a sequence of $A$-valued measurable processes $(u_t^n)_{t\in[0,T]}$ such that the corresponding strict control measures
\[
q^n(dt,da)=\delta_{u_t^n}(da)dt
\]
converge to $q(dt,da)$ in the stable topology, i.e., $q^n\xrightarrow[]{n\to\infty}q$ in \(\mathcal V\). Equivalently, for every bounded continuous function
\[
\varphi:[0,T]\times A\to\R,
\]
one has
\[
\int_0^T\int_A\varphi(t,a)\delta_{u_t^n}(da)dt
\xrightarrow[]{n\to\infty}
\int_0^T\int_A\varphi(t,a)q_t(da)dt.
\]
If $q$ is an adapted random relaxed control, then the processes $u^n$ can be chosen to be adapted.

Furthermore, for every bounded continuous function $\varphi:[0,T]\times A\to\R$,
\[
\sup_{0\le r\le T}\left|\int_0^r\varphi(t,u_t^n)dt-
\int_0^r\int_A\varphi(t,a)q_t(da)dt\right|\xrightarrow[]{n\to\infty}0.
\]
\end{lemma}

\begin{remark}
In this paper, Prokhorov's theorem is used to guarantee tightness of the relaxed control set, and lower semicontinuity is used to ensure that the cost functional does not jump downward under limiting measures, thereby guaranteeing the existence of an optimal relaxed control.
\end{remark}
\end{document}